\documentclass[11pt]{article}
\usepackage[margin=1in]{geometry}
\usepackage{amsmath,amssymb,amsthm}
\usepackage[hidelinks]{hyperref}
\usepackage{graphicx}
\usepackage{booktabs,longtable,array}

\newtheorem{theorem}{Theorem}
\newtheorem{lemma}[theorem]{Lemma}
\newtheorem{corollary}[theorem]{Corollary}
\newtheorem{remark}[theorem]{Remark}
\newcommand{\R}{\mathbb{R}}

\title{The Ekeland--Hofer--Zehnder Capacity of rotated $L_p$-Ellipsoids}

\author{Vardan Oganesyan\thanks{supported by the Higher Education and Science Committee of the
MESCS RA under Research Project No. 25RL-1A040}\\
vardanmath@gmail.com}

\date{}

\begin{document}
\maketitle

\begin{abstract}
We find explicit formulas for the EHZ and cylindrical capacities of rotated $L_{p}$-ellipsoids. In addition, we provide an algorithm to estimate the EHZ capacity.
\end{abstract}

\tableofcontents

\section{Introduction and results}

The initial goal of the author was to code approximation of EHZ capacity for a given domain. Only after  numerical estimates, the author found out explicit formulas proved in this paper. In the last section we explain how we numerically estimate EHZ capacity and give a link to the code with numerical results. 

We assume that coordinates of $\R^{2n}$ are ordered as $(x_1,\ldots,x_n,y_1,\ldots,y_n)$ and the standard symplectic form is
\begin{equation*}
\omega=\sum_{j=1}^n dx_j\wedge dy_j.
\end{equation*}
For $1<p<\infty$, let $q=p/(p-1)$. Consider $B,D\in GL(n,\R)$ and define the $L_p$-ellipsoid
\begin{equation*}
K_p(B,D)=\left\{(x,y)\in\R^{2n}:\|Bx\|_p^p+\|Dy\|_p^p\leq1\right\}.
\end{equation*}
Define the following matrix and its norm:
\begin{equation*}
C=BD^T,\qquad M=\|C\|_{\ell_q^n\to\ell_p^n}:=\max_{\xi\neq0}\frac{\|C\xi\|_p}{\|\xi\|_q}=\max_{\|\xi\|_q=1}\|C\xi\|_p.
\end{equation*}
When $B=D=I$ we get the ordinary ball $B_p^{2n}$. Let $c_0(p)$ be the area of $B_p^2$. 

\begin{lemma}\label{area}
For any $1 \leq p<\infty$,
\begin{equation*}
c_0(p)=Area(B_p^2)=\frac{4\Gamma(1+1/p)^2}{\Gamma(1+2/p)},
\end{equation*}
where $\Gamma$ is the Gamma function.
\end{lemma}

\begin{proof}
In the first quadrant, the area of $B_p^2$ is the area under the graph $t=(1-s^p)^{1/p}$, where $0\leq s\leq1$. Therefore, using $u=s^p$, we get
\begin{equation*}
c_0(p)=4\int_0^1(1-s^p)^{1/p}\,ds=\frac4p\int_0^1u^{1/p-1}(1-u)^{1/p}\,du=\frac4p\mathrm B\left(\frac1p,1+\frac1p\right),
\end{equation*}
where $\mathrm B$ is the Beta function. The identities $\mathrm B(a,b)=\Gamma(a)\Gamma(b)/\Gamma(a+b)$ and $\Gamma(1+s)=s\Gamma(s)$ give the required formula.
\end{proof}

\vspace{0.2cm}

We denote by $c_{EHZ}$ and $c_Z$ the EHZ and cylindrical capacities, respectively.

\vspace{0.2cm}

\begin{theorem}\label{main}
Let $n\geq1$, $2\leq p<\infty$, $q=p/(p-1)$ and $B,D\in GL(n,\R)$. With respect to the standard symplectic form $\omega$, we have
\begin{equation*}
c_{EHZ}(K_p(B,D))=c_Z(K_p(B,D))=\frac{c_0(p)}{\|BD^T\|_{\ell_q^n\to\ell_p^n}}=\frac{4\Gamma(1+1/p)^2}{\Gamma(1+2/p)\|BD^T\|_{\ell_q^n\to\ell_p^n}},
\end{equation*}
where $c_{EHZ}$ and $c_Z$ are the EHZ and cylindrical capacities, respectively. Since the cylindrical capacity is the largest normalized symplectic capacity, every normalized symplectic capacity $c$ satisfies
\begin{equation*}
c(K_p(B,D))\leq c_Z(K_p(B,D))=\frac{c_0(p)}{\|BD^T\|_{\ell_q^n\to\ell_p^n}}.
\end{equation*}
\end{theorem}

\begin{lemma}\label{diagonal}
Let $T=diag(d_1,\ldots,d_n)$ with $d_i\geq0$. Then, for $2\leq p<\infty$ and $q=p/(p-1)$, we have
\begin{equation*}
\|T\|_{\ell_q^n\to\ell_p^n}=\max_i d_i.
\end{equation*}
\end{lemma}

\begin{proof}
Since $q\leq p$, for every $\xi\in\R^n$ we have
\begin{equation*}
\|T\xi\|_p\leq\left(\max_i d_i\right)\|\xi\|_p\leq\left(\max_i d_i\right)\|\xi\|_q.
\end{equation*}
Equality is obtained by taking $\xi=e_j$, where $d_j=\max_i d_i$.
\end{proof}

\begin{corollary}
Let $2\leq p< \infty$ and $\alpha_i,\beta_i>0$. Let
\begin{equation*}
B=diag(\alpha_1^{1/p},\ldots,\alpha_n^{1/p}),\qquad D=diag(\beta_1^{1/p},\ldots,\beta_n^{1/p}).
\end{equation*}
Then
\begin{equation}\label{Haim}
c_{EHZ}(K_p(B,D))=c_0(p)\min_{1\leq i\leq n}(\alpha_i\beta_i)^{-1/p}.
\end{equation}
Note that for these $B,D$, our $L_p$-ellipsoid has the following form:
\begin{equation*}
K_p(B,D)=\left\{(x,y)\in\R^{2n}:\sum_{i=1}^n\left(\alpha_i|x_i|^p+\beta_i|y_i|^p\right)\leq1\right\}.
\end{equation*}
\end{corollary}

\begin{proof}
Since $BD^T=diag((\alpha_1\beta_1)^{1/p},\ldots,(\alpha_n\beta_n)^{1/p})$, this follows directly from Theorem~\ref{main} and Lemma~\ref{diagonal}.
\end{proof}

\begin{remark}
Formula~\eqref{Haim} also follows from the symplectic $p$-product formula of Haim-Kislev and Ostrover \cite[Proposition~1.5]{HaimKislevOstrover}. Each planar factor has area $c_0(p)(\alpha_i\beta_i)^{-1/p}$, and for $p\geq2$ the EHZ capacity of the symplectic $p$-product is the minimum of the capacities of its factors.
\\
Let us note that this type of domain is also studied y Brocic in \cite{Brocic}.
\end{remark}

\vspace{0.5cm}

A Hadamard matrix of order $n$ is a matrix
$A\in\{-1,1\}^{n\times n}$ satisfying
\begin{equation*}
AA^T=nI_n.
\end{equation*}
In this paper, by a normalized Hadamard matrix we mean the orthogonal
matrix $H=n^{-1/2}A$. This means that
\begin{equation*}
H\in O(n),
\qquad
|H_{ij}|=n^{-1/2}.
\end{equation*}
The first general construction of such matrices was given by Sylvester in 1867 \cite{Sylvester1867}. Hadamard studied them in 1893 in connection with the maximal determinant problem \cite{Hadamard1893}, and Paley later gave important constructions using finite fields \cite{Paley1933}. For $n>2$, a Hadamard matrix can exist only if $n$ is divisible by $4$. The Hadamard conjecture asks whether this necessary condition is also sufficient. This conjecture is still open.

The normalized Hadamard matrices of orders $2$ and $4$ are
\begin{equation*}
H_2=\frac{1}{\sqrt{2}}
\begin{pmatrix}
1&1\\
1&-1
\end{pmatrix}
\end{equation*}
and
\begin{equation*}
H_4=\frac{1}{2}
\begin{pmatrix}
1&1&1&1\\
1&-1&1&-1\\
1&1&-1&-1\\
1&-1&-1&1
\end{pmatrix}.
\end{equation*}

The following theorem shows that Hadamard matrices  play an important role in capacities of  $L_p$-ellipsoids.

\begin{theorem}\label{hadamard-capacity}
Let $2 < p< \infty$ and $B,D \in O(n)$. Then
\begin{equation*}
c_0(p)\leq c_{EHZ}(K_p(B,D))=c_Z(K_p(B,D))\leq c_0(p)n^{1/2-1/p}.
\end{equation*}
The upper bound inequality is equality if and only if $BD^T$ is normalized Hadamard matrix.
\end{theorem}

Viterbo's volume-capacity conjecture, formulated in \cite{Viterbo2000}, says that every convex body $K\subset\mathbb{R}^{2n}$ satisfies
\begin{equation*}
c(K)^n\leq n!Vol(K)
\end{equation*}
for every normalized symplectic capacity $c$. There are some results proving and disproving this inequality. Artstein-Avidan, Milman and Ostrover \cite{ArtsteinMilmanOstrover2008} proved the inequality up to a universal multiplicative constant, independent of the dimension. More recently, Abbondandolo, Benedetti and Edtmair \cite{AbbondandoloBenedettiEdtmair} proved that all normalized symplectic capacities coincide on smooth domains sufficiently $C^2$-close to the Euclidean ball. As a result, Viterbo's inequality holds in this neighborhood. Cristofaro-Gardiner and Hind \cite{CristofaroGardinerHind} proved that all normalized symplectic capacities coincide on monotone toric domains in every dimension, extending the four-dimensional result of Gutt, Hutchings and Ramos.

In 2024, Haim-Kislev and Ostrover \cite{HaimKislevOstroverCounterexample} disproved the general conjecture. Their four-dimensional example is the Lagrangian product $P\times RP$, where $P$ is a regular pentagon and $R$ is a rotation through $\pi/2$. They showed that
\begin{equation*}
\frac{c_{EHZ}(P\times RP)^2}{2Vol(P\times RP)} = \frac{3+\sqrt5}{5}>1.
\end{equation*}

Despite this counterexample, the inequality remains valid for many families. It is still unclear if the inequality holds for centrally-symmetric domains. Note that the counterexample is not centrally-symmetric. In 2026 Balitskiy, Mitrofanov and Polyanskii \cite{BalitskiyMitrofanovPolyanskii} proved that
\begin{equation*}
c_{EHZ}(K\times Q)^2 \leq 2 Area(K) Area(Q)
\end{equation*}
when $K\subset\mathbb{R}^2$ is a convex body and $Q\subset\mathbb{R}^2$ is a convex quadrilateral. 

For $L_p-$balls and actions of $O(n)$ we prove the following theorem.

\begin{theorem}\label{viterbo}
Let  $2 < p<\infty$, $n > 1$ and let $B, D \in O(n)$. Then $K_p(B,D)$ satisfies Viterbo's inequality for every normalized capacity. In other words, for every normalized symplectic capacity $c$ we have
\begin{equation*}
c(K_p(B,D))^n < n!Vol(K_p(B,D)).
\end{equation*}
Note that $p$ is strictly greater than 2 and the Viterbo inequality here is also strict. If $p = 2$, then we have equality.
\end{theorem}

Let $H_n$ be a normalized Hadamard matrix. Taking $B=H_n$ and $D=I$, we have $BD^T=H_n$. Therefore,
\begin{equation*}
c_{EHZ}(K_p(H_n,I)) = c_0(p)n^{1/2-1/p}.
\end{equation*}
Recall that $c_0(p)$ is the area of 2-dimensional $B_p^2-$ball. When $p$ is big $B_p^2$ becomes close to $2D-$ box of size $2$. This means that $c_0(p) \to 4$ as $p \to \infty$. It follows that
\begin{equation*}
\lim_{p\to\infty}c_{EHZ}(K_p(H_n,I)) = 4 \sqrt{n}.
\end{equation*}
Also, as $p\to\infty$, the bodies $K_p(H_n,I)$ converge to
\begin{equation*}
K_\infty(H_n,I) = \left\{ (x,y)\in\R^{2n}: \|H_nx\|_\infty \leq 1, \;  \|y\|_\infty \leq 1 \right\}.
\end{equation*}

So, the limiting body is the cube $[-1,1]^{2n}$ with $H_n^T$ matrix applied to its $x$-coordinates. This means that
\begin{equation*}
(2n)^{-1/p} K_\infty(H_n,I) \subset K_p(H_n,I) \subset K_\infty(H_n,I)
\end{equation*}
Monotonicity tells us that
\begin{equation*}
c_{EHZ}(K_p(H_n,I)) \leq c_{EHZ}(K_\infty(H_n,I)) \leq (2n)^{2/p} c_{EHZ}(K_p(H_n,I)).
\end{equation*}
Taking the limit, we get
\begin{equation*}
c_{EHZ}(K_\infty(H_n,I))=4\sqrt n.
\end{equation*}

\begin{figure}[ht]
\centering
\includegraphics[width=0.6\textwidth]{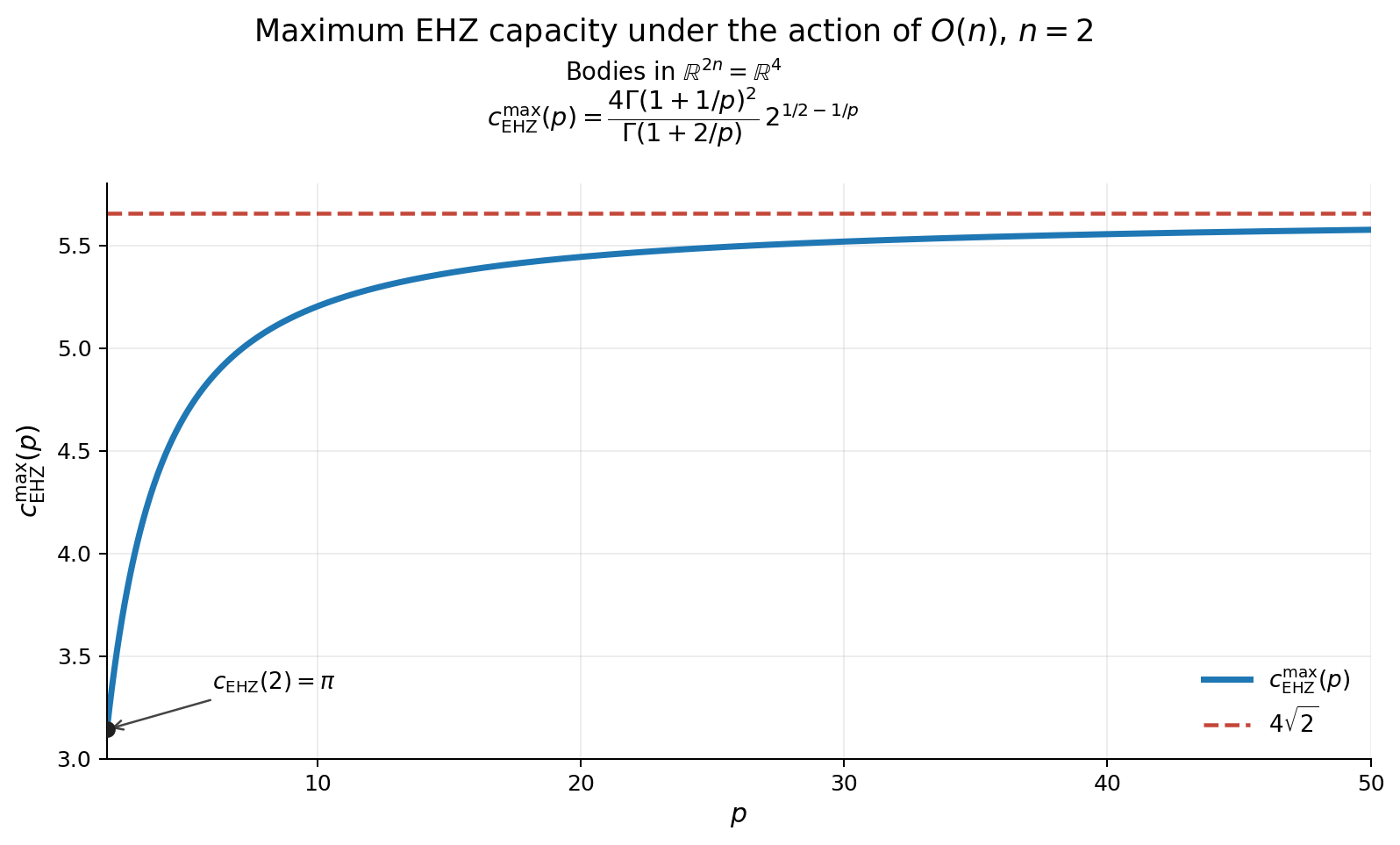}
\hfill
\includegraphics[width=0.6\textwidth]{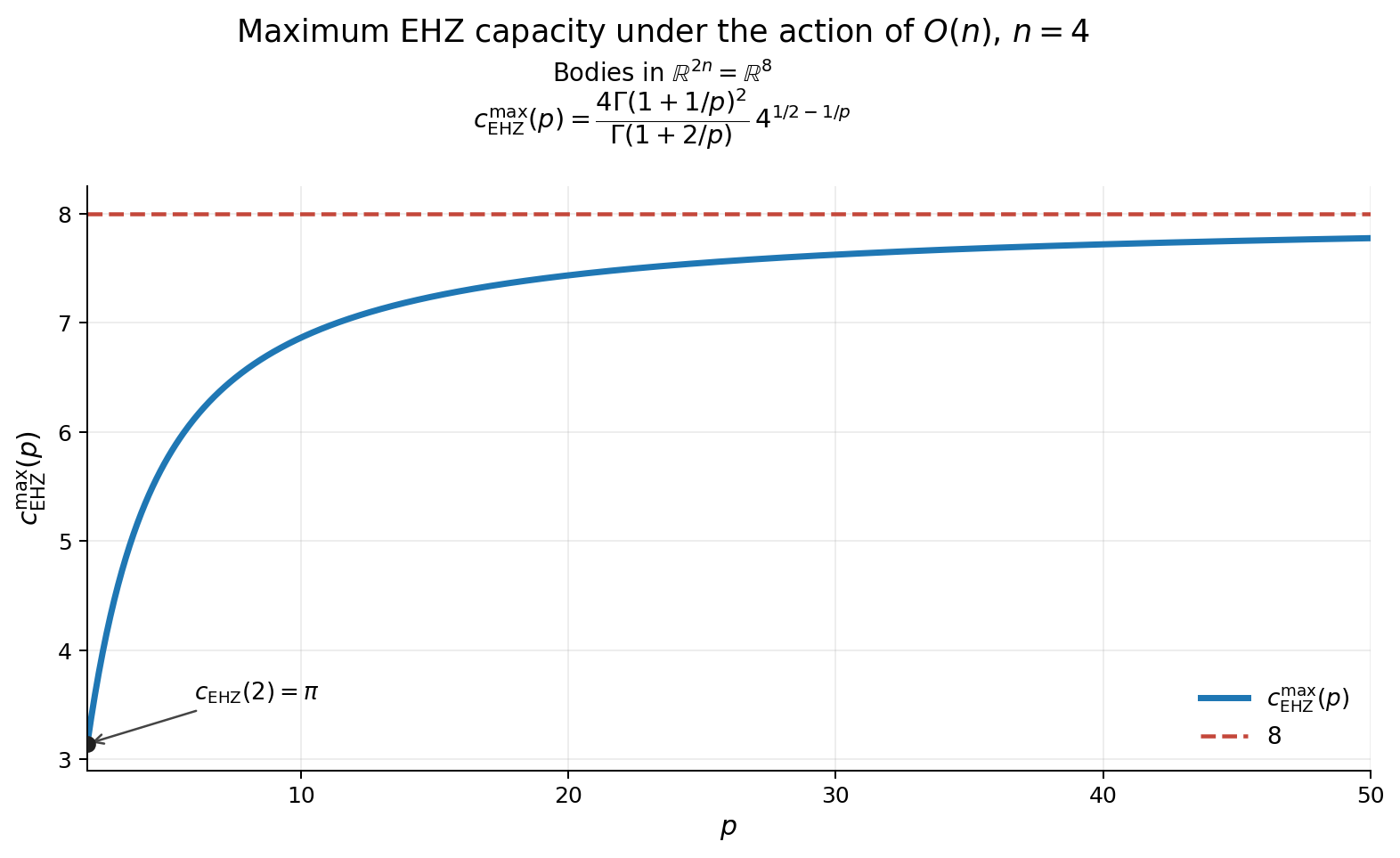}
\caption{The maximal EHZ capacity for orthogonal $B,D$ when $n=2$ and
$n=4$. In both cases the maximum is obtained when $BD^T$ is a normalized
Hadamard matrix. The dashed lines show the limiting capacities
$4 \sqrt{2}$ and $8$, respectively.}
\label{hadamard-capacities}
\end{figure}

The volume of the limiting cube is $Vol(K_{\infty} (H_n, I) ) = 2^{2n} = 4^n$. So, Viterbo's inequality for this cube becomes
\begin{equation*}
(4\sqrt n)^n = 4^n n^{n/2} \leq 4^n n! = n! Vol(K_{\infty}(H_n, I) ).
\end{equation*}
So, in this case for $2n=4$ Viterbo's inequality becomes equality
\begin{equation*}
c_{EHZ}(K_{\infty}(H_2, I) )^2 = (4\sqrt2)^2 = 32 =
2!Vol(K_{\infty}(H_2, I) ).
\end{equation*}
However, for $2n = 8$, there is no equality
\begin{equation*}
c_{EHZ}(K_{\infty}(H_4,I) )^4 = 8^4 = 4096 < 6144 = 4!Vol(K_{\infty} (H_4,I) ).
\end{equation*}

The fact that EHZ capacity for limiting cube in $\R^4$ gives equality in Viterbo's inequality leads to the following idea. Assume that the smooth domain is very close to the cube. Then by perturbing the cube we may get centrally symmetric domain that provides counterexample to Viterbo's conjecture. The author considered 2 options:
\vspace{0.2cm}
- perturb powers in $L_p$-ellipsoid and assume that not all powers are equal.
\\
- truncate a vertex or an edge or a face taking out a small part of the cube. This may decrease the volume but remain EHZ capacity.
\\

Unfortunately, the author could not find counterexample using any of these two ideas.
\\
\\
\textbf{Remark.} The results presented below are obtained using numerical approximation. More details of how this approximation is obtained can be found in Appendix. It is already mentioned that exact formulas above was obtained using numerical estimates and rigorously proved after. Also, tests (can be found in Appendix) show that approximations work. So, there is a strong reason to believe the results below.

\vspace{0.3cm}

Let us focus on dimension $2n = 4$. Our $L_p$-ball has the following form
\begin{equation*}
B_p^4 = \{ (x_1, x_2, y_1, y_2) \in R^4 \; | \; | x_1 |^p + | x_2 |^p + | y_1 |^p + | y_2 |^p \leq 1 \},
\end{equation*}

\noindent

\textbf{NOTATION!!!} Now we change our coordinates and assume that coordinates of $\R^4$ are ordered in the following way:
\begin{equation*}
(x_1, y_1, x_2, y_2),
\end{equation*}
where $(x_1, y_1), (x_2, y_2)$ are symplectic pairs. The reason for this is that the code (which we use for numerical estimates) is written in terms of this coordinates. However, all formulas above look ugly in this coordinates. The author reailzed that this order of coordinates is inconvenient only after writing the code and all approximations. 

\vspace{0.3cm}

Now we want to study what happens with $EHZ$ capacity of $A(B_p^4)$, where $A \in SO(4)$. Previously, we found exact formulas for the action of $SO(2)$
\begin{equation*}
(x_1, x_2) \to C(x_1, X_2), \quad (y_1, y_2) \to D(y_1, y_2) \qquad C, D \in SO(2),
\end{equation*}
not $SO(4)$.

First, note that $EHZ$ capacity is invariant under action of $U(2)$. This means that we can study action of $SO(4)/U(2)$ instead. It is known that $SO(4)/U(2)$ can be identified with $S^2$ using quaternions in the following way:
\begin{equation*}
SO(4)/U(2) = \{ q_1i + q_2j + q_3k \; | \; q_1^2 + q_2^2 + q_3^2 = 1, \;\; q_1, q_2, q_3 \in R    \},
\end{equation*}
$i, j, k$ are quaternions. Moreover, $B_p^4$ has a lot of discrete symmetries. Taking the into account we can replace $S^2$ by a spherical triangle
\begin{equation*}
\mathcal{T} = \{ q_1^2 + q_2^2 + q_3^2 = 1 \; | \; q_1 \geq q_2 \geq q_3 \geq 0  \}
\end{equation*}
without loss of generality (if we are interested only in $EHZ$ capacities). 

It can be checked directly that any matrix $A SO(4)$ have a representative in $\mathcal{T}$ of the following form:
\begin{equation*}
A(q_1,q_2,q_3) =
\begin{pmatrix}
1 & 0 & 0 & 0\\
0 & q_1 & q_2 & q_3\\
0 & -q_2 & 1-\dfrac{q_2^2}{1+q_1} & -\dfrac{q_2q_3}{1+q_1}\\
0 & -q_3 & -\dfrac{q_2q_3}{1+q_1} & 1-\dfrac{q_3^2}{1+q_1}
\end{pmatrix}
\end{equation*}
We see that
\begin{equation*}
A(1,0,0) =
\begin{pmatrix}
1 & 0 & 0 & 0\\
0 & 1 & 0 & 0\\
0 & 0 & 1 & 0\\
0 & 0 & 0 & 1
\end{pmatrix}.
\end{equation*}

\begin{equation*}
A( \frac{1}{\sqrt2}, \frac{1}{\sqrt2}, 0 ) = 
\begin{pmatrix}
1 & 0 & 0 & 0\\
0 & \dfrac{1}{\sqrt{2}} & \dfrac{1}{\sqrt{2}} & 0\\
0 & -\dfrac{1}{\sqrt{2}} & \dfrac{1}{\sqrt{2}} & 0\\
0 & 0 & 0 & 1
\end{pmatrix}
\end{equation*}

\begin{equation*}
A ( \frac{1}{\sqrt3}, \frac{1}{\sqrt3}, \frac{1}{\sqrt3} ) =
\begin{pmatrix}
1 & 0 & 0 & 0\\
0 &\frac{1}{\sqrt{3}} & \frac{1}{\sqrt{3}} & \frac{1}{\sqrt{3}}\\
0 & -\frac{1}{\sqrt{3}}& \frac{3+\sqrt{3}}{6} & -\frac{3-\sqrt{3}}{6}\\
0 & -\frac{1}{\sqrt{3}} & -\frac{3-\sqrt{3}}{6} & \frac{3+\sqrt{3}}{6}
\end{pmatrix}.
\end{equation*}

So, we have only small spherical triangle to study. We can sample points of this triangle to find some statistics. Moreover, the number of points may be small. The results presented in the table below.
\\
\\

\begin{figure}[ht]
    \centering
    \makebox[\textwidth][c]{%
        \includegraphics[width=1.15\textwidth]{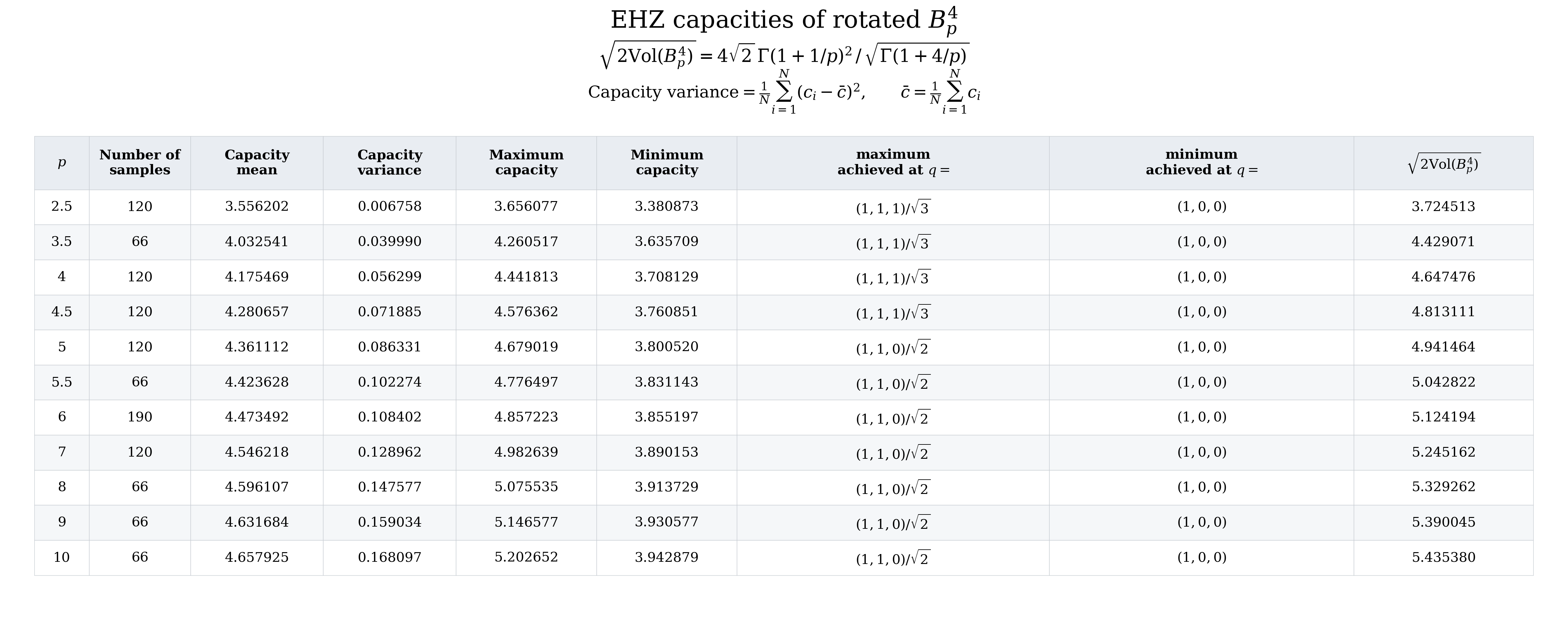}
    }
\end{figure}

Note that $A(1/sqrt{2}, 1/sqrt{2}, 0)$ is 2-dimensional normalized Hadamard matrix applied to $(y_1, x_2)$. Up to symmetry this is equivalent to application of normalized Hadamard matrix to $(x_1, x_2)$ and identity to $(y_1, y_2)$. So, numerical results show that for some $p > p^{*}$ the maximum capacity among all elements of $SO(4)$ provided by Hadamard matrix and inequality Theorem \ref{hadamard-capacity} holds true. However, this is not proved theorem and seen only by numerical approximation.
\\

Let us draw graphs of $A(1/\sqrt{2}, 1/\sqrt{2}, 0)$ and $A(1/\sqrt{3}, 1/\sqrt{3}, 1/\sqrt{3})$ applied to $B_p^4$ for different $p$. This graph also shows that there is $p^{*}$, where maximum starts achieving at $q = (1/\sqrt{2}, 1/\sqrt{2}, 0)$. The author does not have nice geometrical explanation of this fact.

\begin{figure}[ht]
    \centering
    \makebox[\textwidth][c]{%
        \includegraphics[width=1.1\textwidth]{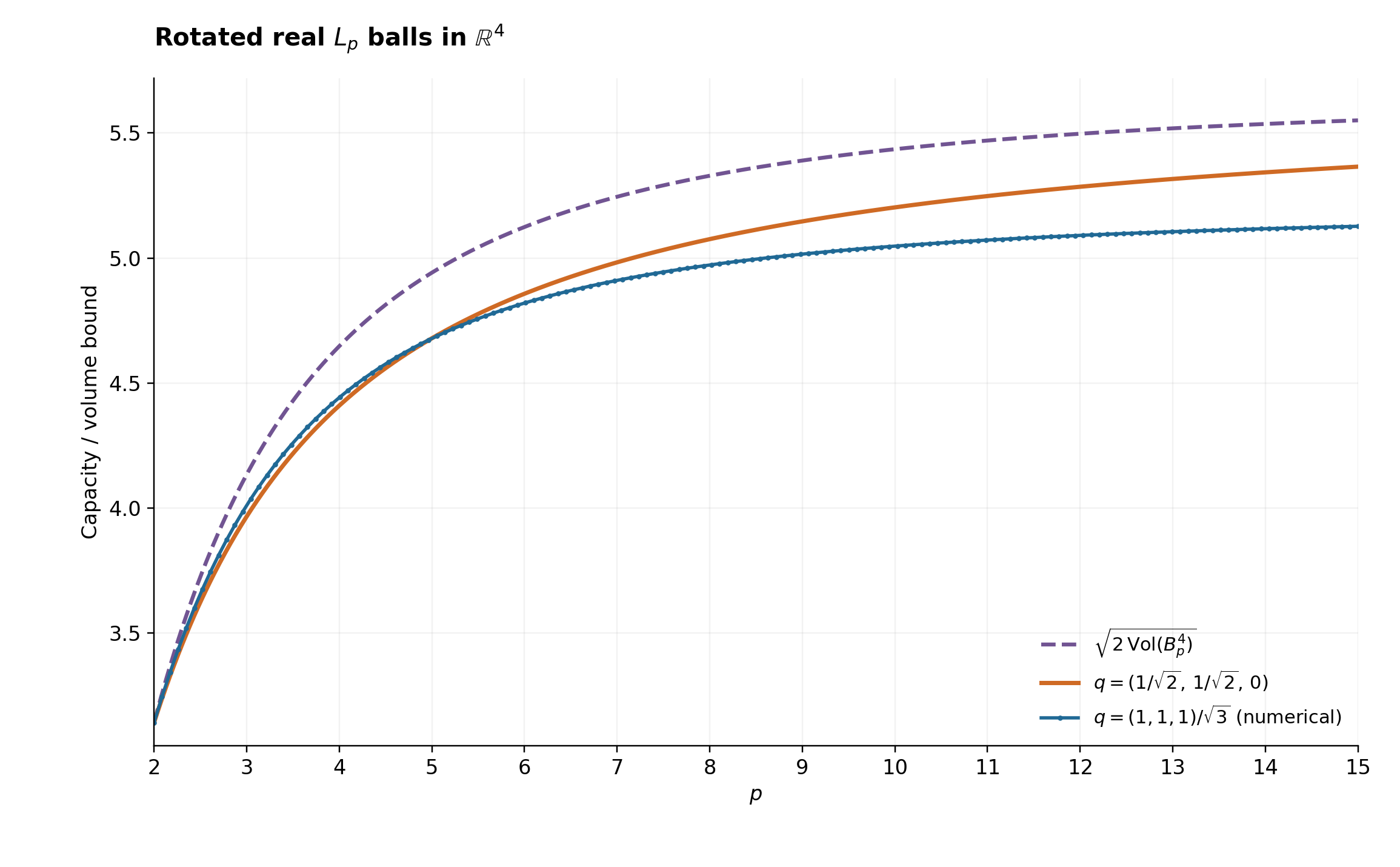}
    }
\end{figure}

\section{Preliminary lemmas}

\subsection{Support function}

Let $K$ be a convex body in $\R^m$. By definition, the support function of $K$ is
\begin{equation*}
h_K(a)=\sup_{z\in K}\langle a,z\rangle.
\end{equation*}
We see that $h_K(a)$ is the largest value of the linear functional $z\mapsto\langle a,z\rangle$ on $K$. Consider the unit $L_p$-ball $B_p^m$. By definition,
\begin{equation*}
h_{B_p^m}(a)=\sup_{\|z\|_p\leq1}\langle a,z\rangle.
\end{equation*}
Since $B_p^m$ is centrally symmetric, replacing $z$ by $-z$ shows that
\begin{equation*}
\sup_{\|z\|_p\leq1}\langle a,z\rangle=\sup_{\|z\|_p\leq1}|\langle a,z\rangle|.
\end{equation*}
This means that the support function coincides with the operator norm $\|a\|_*$ of the linear functional $z\mapsto\langle a,z\rangle$ on $(\R^m,\|\cdot\|_p)$. So,
\begin{equation*}
h_{B_p^m}(a)=\|a\|_*=\sup_{\|z\|_p\leq1}|\langle a,z\rangle|.
\end{equation*}
Let us show that the operator norm is the $\ell_q$-norm, where $1/p+1/q=1$. H\"older's inequality says that for every $z\in B_p^m$ we have
\begin{equation*}
|\langle a,z\rangle|=\left|\sum_{j=1}^m a_jz_j\right|\leq\left(\sum_{j=1}^m|a_j|^q\right)^{1/q}\left(\sum_{j=1}^m|z_j|^p\right)^{1/p}\leq\|a\|_q.
\end{equation*}
Taking the supremum over $z\in B_p^m$ gives
\begin{equation*}
h_{B_p^m}(a)=\|a\|_*\leq\|a\|_q.
\end{equation*}
Let us show that we have equality. Equality for $a=0$ is obvious. For $a\neq0$, consider
\begin{equation*}
z=\frac{1}{\|a\|_q^{q-1}}\bigl(|a_1|^{q-2}a_1,\ldots,|a_m|^{q-2}a_m\bigr).
\end{equation*}
Since $(q-1)p=q$, we have
\begin{equation*}
\|z\|_p^p=\frac{\sum_{j=1}^m|a_j|^{(q-1)p}}{\|a\|_q^{(q-1)p}}=\frac{\sum_{j=1}^m|a_j|^q}{\|a\|_q^q}=1.
\end{equation*}
Therefore, $z\in\partial B_p^m$. Moreover,
\begin{equation*}
\langle a,z\rangle=\frac{\sum_{j=1}^m a_j|a_j|^{q-2}a_j}{\|a\|_q^{q-1}}=\frac{\sum_{j=1}^m|a_j|^q}{\|a\|_q^{q-1}}=\|a\|_q.
\end{equation*}
As a result, we have proved that
\begin{equation}\label{dual_support}
h_{B_p^m}(a)=\|a\|_*=\|a\|_q.
\end{equation}
Let us apply this fact to $K_p(B,D)$. Define the invertible linear map
\begin{equation*}
F(x,y)=(Bx,Dy).
\end{equation*}
By definition, $K_p(B,D)=F^{-1}B_p^{2n}$. Changing variables $Z=Fz$, we obtain
\begin{equation*}
h_{K_p(B,D)}(a)=\sup_{z\in F^{-1}B_p^{2n}}\langle a,z\rangle=\sup_{Z\in B_p^{2n}}\langle a,F^{-1}Z\rangle.
\end{equation*}
Since $\langle a,F^{-1}Z\rangle=\langle (F^{T})^{-1} a,Z\rangle$, we obtain
\begin{equation*}
h_{K_p(B,D)}(a)=\sup_{Z\in B_p^{2n}}\langle (F^{T})^{-1} a,Z\rangle=h_{B_p^{2n}}((F^{T})^{-1} a).
\end{equation*}
Let $a=(\xi,\eta)$, where $\xi,\eta\in\R^n$. Using ~\eqref{dual_support} we get
\begin{equation*}
h_{K_p(B,D)}(\xi,\eta)=\bigl\|(B^{-T}\xi,D^{-T}\eta)\bigr\|_q=\left(\|B^{-T}\xi\|_q^q+\|D^{-T}\eta\|_q^q\right)^{1/q}.
\end{equation*}
Therefore, this support function is the dual norm of $(x,y)\mapsto(\|Bx\|_p^p+\|Dy\|_p^p)^{1/p}$. We have proved the following lemma.

\begin{lemma}\label{lem:support}
The support function of $K_p(B,D)$ satisfies
\begin{equation}\label{support}
h_{K_p(B,D)}(\xi,\eta)^q=\|(B^{T})^{-1} \xi\|_q^q+\|(D^{T})^{-1} \eta\|_q^q,\qquad \xi,\eta\in\R^n.
\end{equation}
\end{lemma}

\subsection{Clarke's dual-action formula}

Let $S^1=\R/(2\pi\mathbb Z)$. For a periodic loop $z\in W^{1,2}(S^1;\R^{2n})$, its symplectic action is
\begin{equation*}
\mathcal A(z)=\frac12\int_{S^1}\omega(z(t),\dot z(t))\,dt.
\end{equation*}
Consider the space
\begin{equation*}
\mathcal E_n=\left\{z\in W^{1,2}(S^1;\R^{2n}):\int_{S^1}z(t)\,dt=0\right\}.
\end{equation*}
The loops in this space are not required to lie on $\partial K$.

\begin{lemma}[Clarke's dual-action formula]\label{lem:clarke}
Let $K \subset\R^{2n}$ be a compact convex body with $C^2$-smooth boundary containing the origin as interior point. For $r>1$, we have
\begin{equation}\label{clarke}
c_{EHZ}(K)^{r/2}=\frac{\pi^r}{2\pi}\min_{\substack{z\in\mathcal E_n\\ \mathcal A(z)=1}}\int_{S^1}h_K(\dot z(t))^r\,dt.
\end{equation}
\end{lemma}

\begin{proof}
This is Clarke's dual-action principle \cite{Clarke}, in the normalization of \cite[Equation~(4)]{HaimKislevOstrover}. See also \cite[Proposition~2.1]{ArtsteinAvidanOstrover}. 
\end{proof}

\subsection{Wirtinger inequality}

\begin{lemma}\label{wirtinger}
Let $1<p<\infty$ and $q=p/(p-1)$. Suppose that $f\in W^{1,q}(S^1;\R)$ satisfies
\begin{equation*}
\int_{S^1}|f(t)|^{p-2}f(t) dt=0,
\end{equation*}
where $S^1 = \R/2 \pi \mathbb{Z}$. Then
\begin{equation*}
\left(\int_{S^1}|f(t)|^p\,dt\right)^{1/p}\leq\frac{\pi^{2/p}}{c_0(p)}\left(\int_{S^1}|f'(t)|^q\,dt\right)^{1/q}.
\end{equation*}
\end{lemma}

\begin{proof}
We apply the generalized Wirtinger inequality of Croce and Dacorogna \cite[Theorem~1.1]{CroceDacorogna}. Let $a,b>1$ and let $d=a/(a-1)$. They compute the constant (here $u'$ is the derivative of $u$)
\begin{equation*}
\alpha(a,b,b)=\inf\left\{\frac{\|u'\|_{L^a(-1,1)}}{\|u\|_{L^b(-1,1)}}:u\in W^{1,a}(\R/2\mathbb{Z}),\ u\neq0,\ \int_{-1}^1|u|^{b-2}u\,ds=0\right\}
\end{equation*}
as follows:
\begin{equation*}
\alpha(a,b,b)=2\left(\frac1d\right)^{1/b}\left(\frac1b\right)^{1/d}\left(\frac2{d+b}\right)^{1/a-1/b}\mathrm B\left(\frac1d,\frac1b\right).
\end{equation*}
Taking $a=q$ and $b=p$, we have $d=p$. Therefore,
\begin{equation*}
\alpha(q,p,p)=\frac2p\mathrm B\left(\frac1p,\frac1p\right)=\frac{2\Gamma(1/p)^2}{p\Gamma(2/p)}=c_0(p),
\end{equation*}
where $c_0(p)$ is computed in Lemma ~\ref{area}. This means that
\begin{equation*}
\left(\int_{-1}^{1}|u|^p\,dt\right)^{1/p} \leq \frac{1}{c_0(p)} \left(\int_{-1}^{1}|u'(t)|^q dt\right)^{1/q}.
\end{equation*}
It remains to transfer the inequality from $\R/2\mathbb Z$ to $\R / 2 \pi \mathbb{Z}$. Change of variable proves the Lemma.
\end{proof}

\begin{lemma}\label{lem:vector-wirtinger}
Let $2\leq p<\infty$, $q=p/(p-1)$ and $u=(u_1,\ldots,u_n)\in W^{1,q}(S^1;\R^n)$. There exists a unique vector $a=(a_1,\ldots,a_n)\in\R^n$ such that
\begin{equation*}
\int_{S^1}|u_i(t)-a_i|^{p-2}(u_i(t)-a_i)\,dt=0,\qquad i=1,\ldots,n.
\end{equation*}
For this vector $a$, we have
\begin{equation}\label{vector-wirtinger}
\left(\int_{S^1}\|u(t)-a\|_q^p\,dt\right)^{1/p}\leq\frac{\pi^{2/p}}{c_0(p)}\left(\int_{S^1}\|\dot u(t)\|_q^q\,dt\right)^{1/q}.
\end{equation}
\end{lemma}

\begin{proof}
First, let us construct the vector $a$. For each $i=1,\ldots,n$, consider
\begin{equation*}
G_i(s)=\int_{S^1}|u_i(t)-s|^p dt,   \qquad    s  \in \R.
\end{equation*}
For fixed $s$ this is the $p$-th power of the $L_p$-distance from $u_i$ to a constant function $s$ on $S^1$. Note that for $p>1$, this function is strictly convex. Also, the reversed triangle inequality gives
\begin{equation*}
G_i(s)^{1/p} \geq | (2\pi)^{1/p}|s|-\|u_i\|_{L^p(S^1)} |.
\end{equation*}
Therefore, $G_i(s)\to +\infty$ as $|s|\to \infty$. This means that $G_i$ has a unique minimum $a_i$. Differentiating, we obtain
\begin{equation*}
0=G_i'(a_i)=-p\int_{S^1}|u_i(t)-a_i|^{p-2}(u_i(t)-a_i)\,dt.
\end{equation*}
This constructs the required vector $a$. Applying Lemma~\ref{wirtinger} to $u_i-a_i$, we get
\begin{equation*}
\|u_i-a_i\|_{L^p(S^1)}\leq\frac{\pi^{2/p}}{c_0(p)}\|\dot u_i\|_{L^q(S^1)}.
\end{equation*}
Since $p/q=p-1\geq1$, the triangle inequality in $L^{p/q}(S^1)$ gives
\begin{equation*}
\left(\int_{S^1}\|u-a\|_q^p\,dt\right)^{q/p}=\left\|\sum_{i=1}^n|u_i-a_i|^q\right\|_{L^{p/q}}\leq\sum_{i=1}^n\|u_i-a_i\|_{L^p}^q.
\end{equation*}
As a result,
\begin{equation*}
\left(\int_{S^1}\|u-a\|_q^p\,dt\right)^{q/p}\leq\left(\frac{\pi^{2/p}}{c_0(p)}\right)^q\sum_{i=1}^n\|\dot u_i\|_{L^q}^q=\left(\frac{\pi^{2/p}}{c_0(p)}\right)^q\int_{S^1}\|\dot u\|_q^q\,dt.
\end{equation*}
Taking the $q$-th root proves~\eqref{vector-wirtinger}.
\end{proof}

\subsection{Cylindrical upper bound}

In the coordinate order $(x_1,\ldots,x_n,y_1,\ldots,y_n)$, the first symplectic coordinate pair is $(x_1,y_1)$. The standard cylinder of radius $R$ is
\begin{equation*}
Z_R=\left\{(x,y)\in\R^{2n}:x_1^2+y_1^2<R^2\right\}.
\end{equation*}
Recall that $c_Z(K)$ is the infimum of $\pi R^2$ over symplectic embeddings of a neighborhood of $K$ into $Z_R$. Thus, normalization and monotonicity imply $c(K)\leq c_Z(K)$ for every normalized symplectic capacity $c$.

\begin{lemma}\label{lem:cylindrical-upper}
Let $1<p<\infty$, $q=p/(p-1)$ and $B,D\in GL(n,\R)$. Set $C=BD^T$ and $M=\|C\|_{\ell_q^n\to\ell_p^n}$. Then
\begin{equation}\label{eq:cylindrical-upper}
c_Z(K_p(B,D))\leq\frac{c_0(p)}M.
\end{equation}
Consequently, the same upper bound holds for every normalized symplectic capacity, including $c_{EHZ}$.
\end{lemma}

\begin{proof}

Choose $v\in\R^n$ such that $\|v\|_q=1$ and $\|Cv\|_p=M$. Such a vector exists due to compactness of $\ell_q-$sphere. Since $C$ is invertible, we get $M>0$. Consider
\begin{equation*}
u=\frac{|Cv|^{p-2}Cv}{M^p}.
\end{equation*}
Since $(p-1)q=p$, we have
\begin{equation*}
\|u\|_q=\frac{\|Cv\|_p^{p-1}}{M^p}=\frac1M,\qquad \langle u,Cv\rangle=\frac{\|Cv\|_p^p}{M^p}=1.
\end{equation*}
Define $\alpha=B^Tu$ and $\beta=D^Tv$. Then
\begin{equation*}
\langle\alpha,\beta\rangle=\langle u,BD^Tv\rangle=1.
\end{equation*}
We want to change variables to make $\langle\alpha,x\rangle$ and $\langle\beta,y\rangle$ the first pair of symplectic coordinates. Let us construct that change of variables explicitly.

Choose a basis $\rho_2,\ldots,\rho_n$ of $\beta^\perp$, and let $A$ be the $n \times n$ matrix with rows $\alpha^T,\rho_2^T,\ldots,\rho_n^T$. These rows are linearly independent because $\alpha \notin \beta^\perp$ ($\alpha$ is not perpendicular to $\beta$). Therefore, $A$ is invertible. We see that
\begin{equation*}
A \beta=e_1 = (1, 0, \ldots, 0)^T,\qquad A^{-1}e_1=\beta.
\end{equation*}
Consider a linear map
\begin{equation*}
F(x,y)=(X,Y),\qquad X=Ax,\qquad Y = (A^{T})^{-1} y.
\end{equation*}
We see that for any $(x,y),(x',y')\in\R^{2n}$
\begin{equation*}
\omega( F(x,y), F(x',y'))=\langle Ax, (A^{T})^{-1} y' \rangle-\langle Ax', (A^{T})^{-1} y \rangle = \langle x,y'\rangle-\langle x',y\rangle.
\end{equation*}
This means that  $F$ is symplectic.

The new coordinates are ordered as $(X_1,\ldots,X_n,Y_1,\ldots,Y_n)$. Their first symplectic pair is $(X_1,Y_1)$ (because $F$ is symplectic). By construction,
\begin{equation*}
\begin{gathered}
X_1=e_1^T A x=\langle \alpha, x \rangle = \langle u,Bx \rangle,
\\
Y_1=e_1^T A^{-T} y=\langle A^{-1} e_1, y \rangle = \langle \beta, y \rangle = \langle v, Dy \rangle.
\end{gathered}
\end{equation*}
Therefore, H\"older's inequality gives
\begin{equation*}
|X_1| \leq \| u \|_q \| Bx \|_p = \frac{1}{M} \|Bx\|_p,\qquad |Y_1| \leq  \| v \|_q \|Dy\|_p =  \|Dy\|_p.
\end{equation*}
For every $(x,y)\in K_p(B,D)$, we obtain
\begin{equation*}
M^p|X_1|^p+|Y_1|^p\leq\|Bx\|_p^p+\|Dy\|_p^p\leq1.
\end{equation*}
Define $\Omega \subset \R^2$
\begin{equation*}
\Omega=\left\{(s,t)\in\R^2:M^p|s|^p+|t|^p\leq1\right\}.
\end{equation*}
We showed that
\begin{equation*}
F(K_p(B,D)) \subset \left\{ (X,Y)\in\R^{2n}:(X_1,Y_1) \in \Omega \right \}.
\end{equation*}
The change of variables $(s,t) \to (Ms,t)$ gives
\begin{equation*}
Area(\Omega)=\frac1MArea(B_p^2)=\frac{c_0(p)}M.
\end{equation*}

Fix $\varepsilon>0$. Choose a smooth convex open domain $U\subset\R^2$ such that $\Omega\subset U$ and $Area(U)<Area(\Omega)+\varepsilon$. Define a number $R$ such that $\pi R^2=Area(U)$. Since $U$ is in $\R^2$, there is symplectic map from $U$ to the ball $B^2(R) \subset \R^2$. Note this map by $\varphi = (\varphi_1, \varphi_2)$.

Define
\begin{equation*}
\varphi(X,Y) = \bigl( \varphi_1(X_1,Y_1),X_2,\ldots,X_n,\varphi_2(X_1,Y_1),Y_2,\ldots,Y_n \bigr).
\end{equation*}
We get that $\varphi \circ F$ is a symplectic embedding of a neighborhood of $K_p(B,D)$ into $Z_R$. By the definition of $c_Z$,
\begin{equation*}
c_Z(K_p(B,D)) \leq \pi R^2=Area(U) < \frac{c_0(p)}{M} + \varepsilon.
\end{equation*}
Taking $\varepsilon$ as small as we want proves the lemma.
\end{proof}

\section{Proof of Theorem~\ref{main}}

Recall that $C=BD^T$ and $M=\|C\|_{\ell_q^n\to\ell_p^n}$. We apply Clarke's formula~\eqref{clarke} with $r=q$. Write
\begin{equation*}
z(t)=(B^Tu(t),D^Tv(t)),\qquad u,v\in W^{1,2}(S^1;\R^n).
\end{equation*}
Since $B$ and $D$ are invertible, the condition $z\in\mathcal E_n$ is equivalent to
\begin{equation}\label{ordinary-means}
\int_{S^1}u(t)\,dt=\int_{S^1}v(t)\,dt=0.
\end{equation}
Formula~\eqref{support} gives
\begin{equation*}
h_{K_p(B,D)}(\dot z(t))^q=\|\dot u(t)\|_q^q+\|\dot v(t)\|_q^q.
\end{equation*}
By the definition of the action,
\begin{equation*}
\mathcal A(z)=\frac12\int_{S^1}\left(\langle u,C\dot v\rangle-\langle\dot u,Cv\rangle\right)\,dt.
\end{equation*}
Since $u,v$ are periodic, integration by parts gives $\int_{S^1}\langle\dot u,Cv\rangle\,dt=-\int_{S^1}\langle u,C\dot v\rangle\,dt$. Therefore, the condition $\mathcal A(z)=1$ becomes
\begin{equation}\label{eq:action-one}
\int_{S^1}\langle u(t),C\dot v(t)\rangle\,dt=1.
\end{equation}
As a result, Clarke's formula has the following form:
\begin{equation}\label{reduced}
c_{EHZ}(K_p(B,D))^{q/2}=\frac{\pi^q}{2\pi}\min\int_{S^1}\left(\|\dot u(t)\|_q^q+\|\dot v(t)\|_q^q\right)\,dt,
\end{equation}
where the minimum is taken over all $u,v\in W^{1,2}(S^1;\R^n)$ satisfying~\eqref{ordinary-means} and~\eqref{eq:action-one}. Since $q \leq 2$, we can apply Lemma~\ref{lem:vector-wirtinger} to $u$. Choose the vector $a$ given by that lemma. Periodicity of $v$ implies
\begin{equation*}
\int_{S^1}\langle a,C\dot v(t)\rangle\,dt=\left\langle a,C\int_{S^1}\dot v(t)\,dt\right\rangle=0.
\end{equation*}
Therefore,
\begin{equation*}
1=\left|\int_{S^1}\langle u(t)-a,C\dot v(t)\rangle\,dt\right|.
\end{equation*}
Define
\begin{equation*}
X=\left(\int_{S^1}\|\dot u(t)\|_q^q\,dt\right)^{1/q},\qquad Y=\left(\int_{S^1}\|\dot v(t)\|_q^q\,dt\right)^{1/q}.
\end{equation*}
For each $t$, H\"older's inequality and the definition of $M$ give
\begin{equation*}
|\langle u(t)-a,C\dot v(t)\rangle|\leq\|u(t)-a\|_q\|C\dot v(t)\|_p\leq M\|u(t)-a\|_q\|\dot v(t)\|_q.
\end{equation*}
Applying H\"older's inequality on $S^1$, and then~\eqref{vector-wirtinger}, we obtain
\begin{equation*}
\begin{gathered}
1\leq M\left(\int_{S^1}\|u(t)-a\|_q^p\,dt\right)^{1/p}Y\leq\frac{M\pi^{2/p}}{c_0(p)}XY, 
\\
X^q+Y^q\geq2(XY)^{q/2}\geq2\left(\frac{c_0(p)}{M\pi^{2/p}}\right)^{q/2}.
\end{gathered}
\end{equation*}
In formula ~\eqref{reduced} the integral equals $X^q + Y^q$. Taking the minimum and using $q-1=q/p$, we get
\begin{equation*}
c_{EHZ}(K_p(B,D))^{q/2}\geq\pi^{q-1}\left(\frac{c_0(p)}{M\pi^{2/p}}\right)^{q/2}=\left(\frac{c_0(p)}M\right)^{q/2}.
\end{equation*}
Finally, we get
\begin{equation*}
c_{EHZ}(K_p(B,D))\geq\frac{c_0(p)}M.
\end{equation*}
By Lemma~\ref{lem:cylindrical-upper},
\begin{equation*}
c_{EHZ}(K_p(B,D))\leq c_Z(K_p(B,D))\leq\frac{c_0(p)}M.
\end{equation*}
Combining the last two inequalities we get
\begin{equation*}
c_{EHZ}(K_p(B,D))=\frac{c_0(p)}M=\frac{4\Gamma(1+1/p)^2}{\Gamma(1+2/p)\|BD^T\|_{\ell_q^n\to\ell_p^n}}.
\end{equation*}
The theorem is proved.

\section{Proof of Theorem \ref{hadamard-capacity}}
\begin{proof}

Recall that a normalized Hadamard matrix is an orthogonal matrix $H\in O(n)$ whose entries satisfy $|H_{ij}|=n^{-1/2}$. We will show that these matrices minimize the norm $\|C\|_{\ell_q^n\to\ell_p^n}$ among orthogonal matrices. By Theorem~\ref{main}, this means that they maximize the capacity.

\begin{lemma}\label{hadamard-norm}
Let $2 < p<\infty$, $q=p/(p-1)$ and $C\in O(n)$. Then
\begin{equation}\label{orthogonal-norm-bounds}
n^{1/p-1/2}\leq\|C\|_{\ell_q^n\to\ell_p^n}\leq1.
\end{equation}
The lower bound is equality if and only if $C$ is a normalized Hadamard matrix. If $p>2$, then equality in the lower bound holds only for normalized Hadamard matrices.
\end{lemma}

\begin{proof}

First, recall that $\|z\|_s\leq\|z\|_r$ whenever $1 \leq r \leq s< \infty$. Since $q \leq 2 \leq p$, we get that for every $\xi \in \R^n$, $C \in O(n)$ we have
\begin{equation*}
\|C\xi\|_p\leq\|C\xi\|_2=\|\xi\|_2\leq\|\xi\|_q.
\end{equation*}
Dividing by $\|\xi\|_q$ and taking the supremum over $\xi\neq0$, we obtain
\begin{equation*}
\|C\|_{\ell_q^n\to\ell_p^n}\leq1.
\end{equation*}
This gives us the upper bound. Let us prove the lower bound.

Applying H\"older's inequality with conjugate for $p/2$ and $p/(p-2)$, we obtain
\begin{equation*}
\|z\|_2^2 = \sum_{i=1}^n |z_i|^2 \cdot 1 \leq \left( \sum_{i=1}^n |z_i|^p \right)^{2/p} \left( \sum_{i=1}^n 1^{p/(p-2)} \right)^{(p-2)/p} = n^{1-2/p} \|z\|_p^2.
\end{equation*}
This implies that
\begin{equation}\label{eq:norm-comparison}
\|z\|_p\geq n^{1/p-1/2}\|z\|_2.
\end{equation}
This inequality turns into equality if and only if $|z_i|^p$ are all equal. 

Let $e_j$ be the standard basis vector. We have $\|e_j\|_q=1$, and $Ce_j$ is the $j$-th column of $C$. Since $C$ is orthogonal, $\|Ce_j\|_2=1$. By the definition of the operator norm and~\eqref{eq:norm-comparison},
\begin{equation*}
\|C\|_{\ell_q^n\to\ell_p^n}\geq\frac{\|Ce_j\|_p}{\|e_j\|_q}=\|Ce_j\|_p\geq n^{1/p-1/2}\|Ce_j\|_2=n^{1/p-1/2}.
\end{equation*}
This proves the lower bound.
\\

Let $H$ be a normalized Hadamard matrix. We first compute two operator norms of $H$. Since every entry of $H$ has absolute value $n^{-1/2}$, for every $\xi\in \R^n$ and every $i$ we have
\begin{equation*}
|(H\xi)_i|=\left|\sum_{j=1}^nH_{ij}\xi_j\right|\leq\sum_{j=1}^n|H_{ij}||\xi_j|=n^{-1/2}\|\xi\|_1.
\end{equation*}
Taking the maximum over $i$ gives $\|H\xi\|_\infty\leq n^{-1/2}\|\xi\|_1$. Equality is obtained by taking $\xi=e_j$ (every entry of $He_j$ has absolute value $n^{-1/2}$). Therefore,
\begin{equation*}
\|H\|_{\ell_1^n\to\ell_\infty^n}=n^{-1/2}.
\end{equation*}
Moreover, $H^*H=H^TH=I$, so $H$ preserves the Euclidean norm also on $\R^n$. Therefore, $\|H\|_{\ell_2^n \to \ell_2^n}=1$.

Let us recall the Riesz--Thorin interpolation theorem (see \cite[Theorem~1.1.1]{BerghLofstrom}). It says that
if a linear operator $T$ satisfies
\begin{equation*}
\|T\|_{\ell_{r_0}^n\to\ell_{s_0}^n}\leq A_0, \qquad \|T\|_{\ell_{r_1}^n\to\ell_{s_1}^n}\leq A_1,
\end{equation*}
then, for $0<\theta<1$,
\begin{equation*}
\|T\|_{\ell_{r_\theta}^n\to\ell_{s_\theta}^n} \leq A_0^{1-\theta}A_1^\theta,
\end{equation*}
where
\begin{equation*}
\frac{1}{r_\theta} = \frac{1-\theta}{r_0}+\frac{\theta}{r_1}, \qquad \frac{1}{s_\theta} = \frac{1-\theta}{s_0} + \frac{\theta}{s_1}.
\end{equation*}
Let us apply the theorem between
\begin{equation*}
H:\ell_1^n \longrightarrow \ell_\infty^n, \qquad H:\ell_2^n\longrightarrow\ell_2^n.
\end{equation*}
Taking $\theta=2/p$, we obtain
\begin{equation*}
\frac{1}{r_\theta} = 1 - \frac{\theta}{2} = 1 - \frac{1}{p} = \frac{1}{q}, \qquad
\frac{1}{s_\theta} = \frac{\theta}{2} = \frac{1}{p}.
\end{equation*}
Therefore,
\begin{equation*}
\|H\|_{\ell_q^n(\R)\to\ell_p^n(\R)} \leq \left(n^{-1/2}\right)^{1-2/p}= n^{1/p-1/2}.
\end{equation*}
Combining this with the already proved lower bound, we get 
\begin{equation*}
\|H\|_{\ell_q^n(\mathbb R)\to\ell_p^n(\mathbb R)} = n^{1/p-1/2}.
\end{equation*}
Let us study matrices when we have equality. Assume that $p>2$ and that
$C\in O(n)$ satisfies
\begin{equation*}
\|C\|_{\ell_q^n\to\ell_p^n}=n^{1/p-1/2}.
\end{equation*}
For every standard basis vector $e_j$, we have
\begin{equation*}
n^{1/p-1/2} \leq \|Ce_j\|_p \leq \|C\|_{\ell_q^n\to\ell_p^n}\|e_j\|_q = n^{1/p-1/2}.
\end{equation*}
This means that equality holds in~\eqref{eq:norm-comparison} for every
column $Ce_j$. Since $p>2$, we have equality if and only if all coordinates have the same absolute value. Since $\|Ce_j\|_2=1$, it follows that
\begin{equation*}
|C_{ij}|=n^{-1/2}, \qquad 1\leq i,j\leq n.
\end{equation*}
This means that $C$ is a normalized Hadamard matrix. Conversely, the argument above shows that every normalized Hadamard matrix has the norm equal to its  lower bound.
\end{proof}

Since $B,D$ are orthogonal, so is $C=BD^T$. Put $M=\|C\|_{\ell_q^n\to\ell_p^n}$. Lemma~\ref{hadamard-norm} gives
\begin{equation*}
n^{1/p-1/2}\leq M\leq1.
\end{equation*}
All quantities are positive. Taking reciprocals reverses the inequalities:
\begin{equation*}
1\leq\frac1M\leq n^{1/2-1/p}.
\end{equation*}
Multiplying by $c_0(p)$ and using Theorem~\ref{main}, we obtain the required capacity bounds.

For $p>2$, the capacity is equal to its upper bound if and only if $M=n^{1/p-1/2}$. By Lemma~\ref{hadamard-norm} this happens if and only if  $BD^T$ is a normalized Hadamard matrix. In particular, if such a matrix $H$ exists, we can take $B=H$ and $D=I$. Then
\begin{equation*}
c_{EHZ}(K_p(H,I))=c_Z(K_p(H,I))=c_0(p)n^{1/2-1/p}.
\end{equation*}
For $p=2$, we have $M=1$ for every orthogonal pair $B,D$, and all these domains have capacity $c_0(2)=\pi$.
\end{proof}

\section{Proof of Theorem \ref{viterbo}}

Let us recall that the cylindrical capacity is the biggest one. Then Theorem~\ref{main} says that for any normalized capacity  we have
\begin{equation}\label{capacity-upper-orthogonal}
c(K_p(B,D))
\leq c_Z(K_p(B,D))
\leq c_0(p)n^{1/2-1/p}.
\end{equation}
Since $B, D$ are orthogonal, direct computations show
\begin{equation}\label{viterbo-volume-side}
\begin{gathered}
Vol(K_p(B,D)) = Vol(B_p^{2n}) = \frac{2^{2n}\Gamma(1+1/p)^{2n}}{\Gamma(1+2n/p)},
\\
\left( n!Vol(K_p(B,D)) \right)^{1/n} = 4\Gamma(1+1/p)^2
\left( \frac{n!}{\Gamma(1+2n/p)} \right)^{1/n}.
\end{gathered}
\end{equation}
We need to compare this expression with the right-hand side of~\eqref{capacity-upper-orthogonal}.
\begin{lemma}
For any  $n > 1$ and $0\leq s\leq1$, we have
\begin{equation}\label{gamma-viterbo}
\Gamma(1+ns) \leq n! \Gamma(1+s)^n n^{-n(1-s)/2}.
\end{equation}
\end{lemma}

\begin{proof}
Define
\begin{equation*}
G(s)=\log( \Gamma(1+ns) ) - n \log( \Gamma(1+s) ).
\end{equation*}
Let us prove that $G$ is convex. By the Weierstrass product for the Gamma function we have 
\begin{equation*}
\frac{1}{\Gamma(x)} = xe^{\gamma_{E}x} \prod_{m=1}^{\infty} \left( 1 + \frac{x}{m} \right)e^{-x/m}, \qquad x>0,
\end{equation*}
where $\gamma_{E}$ is the Euler constant. Taking logarithms gives
\begin{equation*}
\begin{gathered}
-\log(\Gamma(x)) = \log(x) + \gamma_{E}x + \sum_{m=1}^{\infty}\left( \log\left( 1+\frac{x}{m}\right)-\frac{x}{m} \right),
\\
\frac{d}{dx}\log\Gamma(x) = - \gamma_{E} - \frac{1}{x} + \sum_{m=1}^{\infty}  \frac{x}{m(m + x)},
\\
\frac{d^2}{dx^2}\log\Gamma(x) = \frac{1}{x^2} + \sum_{m=1}^{\infty} \frac{1}{(m+x)^2} = \sum_{m=0}^{\infty}\frac{1}{(x+m)^2}.
\end{gathered}
\end{equation*}
Note that each of the series above converges uniformly  on every compact subset of $(0,\infty)$. So, we are allowed to apply termwise differentiation.

Using the formulas above and taking into account that $n > 1$ and elements with $k < n$ are positive, we obtain
\begin{equation*}
G''(s) = \sum_{k=1}^n\sum_{m=0}^{\infty} \left( \frac{1}{(s+m+k/n)^2} - \frac{1}{(s+m+1)^2} \right) > 0.
\end{equation*}
This means that $G$ is convex.

Since $G(0)=0$, $G(1)=\log(n!)$ and $G$ is convex, we have (less than the line joining endpoints)
\begin{equation*}
G(s)\leq s\log(n!).
\end{equation*}
Taking exponents of both sides we have
\begin{equation*}
\Gamma(1+ns) \leq (n!)^s\Gamma(1 + s)^n.
\end{equation*}
Now we need an estimate for $(n!)^s$. Note $(n!)^2 =\prod_{k=1}^n k(n+1-k) \geq n^n$. This implies that $n!\geq n^{n/2}$. As a result,
\begin{equation*}
(n!)^s\leq n! n^{ -n(1-s)/2 }.
\end{equation*}
This proves the lemma.
\end{proof}

\vspace{0.6cm}

Assume that $s=2/p$ in the lemma above. We obtain
\begin{equation*}
\frac{n^{1/2-1/p}}{\Gamma(1+2/p)} \leq \left( \frac{n!}{\Gamma(1+2n/p)} \right)^{1/n}.
\end{equation*}
Using formulas ~\eqref{capacity-upper-orthogonal}, ~\eqref{viterbo-volume-side} and the value of $c_0(p) = \frac{4\Gamma(1+1/p)^2}{\Gamma(1+2/p)}$ we get
\begin{equation*}
c(K_p(B,D)) \leq c_0(p)n^{1/2-1/p} < \left( n!Vol (K_p(B,D)) \right)^{1/n}.
\end{equation*}
This proves the theorem.

\section{Appendix. Code for approximation of  EHZ capacity}

As it is written in the introduction, all formulas are first observed after numerical computations and rigorously proved after. In this section we give the idea of how we estimate EHZ capacity for a given domain. Note that checking the code for many domains, where the answer is known, gives us the correct answer with error around $10^{-6}$.
\\
\\
The code and numerical estimates can be found if you follow this link https://github.com/vardanbobo007/symplectic-capacity/blob/main/README.md
\\
\\
We use the following convention
\begin{equation*}
J=\begin{pmatrix}
0 & -I  \\
I & 0
\end{pmatrix},
\qquad
\omega(u,v)= \langle Ju, v \rangle.
\end{equation*}
We use the function $F$ and the following convex region
\begin{equation*}
K=\{x\in\mathbb{R}^{2n}:F(x)\leq 1\}
\end{equation*}
such that $\nabla F \neq 0$ near $\partial K$. To estimate EHZ capacity, we need to approximate the following problem:
\begin{equation*}
\begin{split}
c_{EHZ}(K)=\inf_{\gamma,\lambda}\Bigg\{
&\frac12\int_0^1
\left\langle-J\gamma(t),\dot\gamma(t)\right\rangle dt:\\
&\gamma(0)=\gamma(1),\qquad F(\gamma(t))=1,\\
&\dot\gamma(t)=-\lambda(t)J\nabla F(\gamma(t)),
\qquad \lambda(t)>0
\Bigg\}.
\end{split}
\tag{*}\label{eq:ehz-problem-recalled}
\end{equation*}

\subsection{Basic idea of gradient optimization}

Let us discuss the basic idea first. Suppose that an object is described by parameters $\theta$ and $L(\theta)$ is a positive real number that measures
how bad this object is, called the loss. A big value $L(\theta)$ means that the object is far from
what we want. A small value means that it is closer.

The gradient $\nabla_\theta L$ points in the direction where the loss increases most quickly. To reduce the
loss we move in the opposite direction. Starting from $\theta^{(0)}$ we repeat
\begin{equation*}
\theta^{(s+1)} = \theta^{(s)}-\eta_s\nabla_\theta L(\theta^{(s)}),
\end{equation*}
where $\eta_s>0$ is a step size. After many updates the coefficients move toward a local minimum.

The loss may have many local minima and the result depends on the starting point. For this reason we optimize many loops simultaneously. We start from many random
loops and later continue with the best ones.

The idea is that some initial loops converge to a local minimum and one of them gives us the EHZ capacity.

\subsection{General form of the loss}

There are three things that we want from every loop. It should lie on the
boundary, its tangent vector should belong to the characteristic direction and
its action should be small. We represent these requirements by
\begin{equation*}
L_{bound},\qquad L_{char},\qquad L_{act}.
\end{equation*}
Their explicit form will be given later. For now, let us discuss the general idea. We combine them into one loss
\begin{equation*}
L(\theta) = \alpha_{bound}L_{bound}(\theta) + \alpha_{\mathrm{char}}L_{char}(\theta) + \alpha_{act} L_{act}(\theta),
\end{equation*}
where all coefficients are positive. They let us give priority to some parts
of the problem. For example, if  loops do not stay close to the boundary we
can increase $\alpha_{\mathrm{bound}}$. If their tangent vectors do not follow
the characteristic direction we can increase
$\alpha_{char}$. The action loss should distinguish different valid
characteristics but it should not dominate before the geometric conditions
are satisfied.

The step size $\eta_s$ decreases during the gradient optimization. We need steps smaller and smaller as we get closer to the local minimum

\subsection{Fourier representation of the loops}

We need to define a map $S^1 \to \mathbb{R}^{2n}$, i.e. to define $2n$ maps $\gamma_1, \ldots, \gamma_{2n}$. We represent each closed loop
\begin{equation*}
\gamma_\ell:[0,1]\longrightarrow \mathbb{R}^{2n}, \qquad \ell = 1,\ldots,B
\end{equation*}
by finite Fourier series. For a fixed number $N$ we write
\begin{equation*}
\gamma_{i,\ell}(t)=a_{i,\ell,0}
+\sum_{k=1}^{N}\left(
a_{i,\ell,k}\cos(2\pi kt)+b_{i,\ell,k}\sin(2\pi kt)
\right),
\qquad i = 1,\ldots,2n.
\end{equation*}
Note that all these loops are automatically closed. All Fourier coefficients are the parameters that we change during optimization.

It is easy to compute tangent vector explicitly:
\begin{equation*}
\dot\gamma_{i,\ell}(t) = \sum_{k=1}^{N}2\pi k\left( -a_{i,\ell,k}\sin(2\pi kt)+b_{i,\ell,k}\cos(2\pi kt) \right).
\end{equation*}
\\
\\
\textbf{Stage 1: global sampling}
\\

The goal of stage~1 is a wide search and its purpose is to find several promising regions in the parameter space.

In the first stage we generate $B$ random Fourier loops. Usually B = 512 is enough when $2n = 4$ or $6$.

Let us describe how we randomly generate one loop. The same works for each loop. The center of each loop $i = 1, \ldots, B$ is sampled as
\begin{equation*}
a_{i,0}\sim\mathcal N(0,\sigma_{\mathrm c}^2I),
\end{equation*}
where $\sigma_c$ is the constant we choose. To construct the first Fourier coefficients we sample a Gaussian vector $z_i$ and
normalize it
\begin{equation*}
u_i=\frac{z_i}{\|z_i\|}, \qquad z_i\sim\mathcal N(0,I).
\end{equation*}
The first cosine coefficient is taken in the direction $u_i$ and
the sine coefficient in the direction $-Ju_i$:
\begin{equation*}
a_{i,1} = R_i^{\cos}u_i, \qquad b_{i,1} = -R_i^{\sin}Ju_i,
\end{equation*}
where real numbers $R_i^{\cos}$ and $R_i^{\sin}$ are sampled independently from a fixed interval (which we choose). The reasons for these choices are that Gaussian distribution gives uniform distribution on a sphere. 

Coefficients for higher terms are sampled using a decaying Gaussian scale. For
$k=2,\ldots,N$ the code uses
\begin{equation*}
a_{i,k},b_{i,k}
\sim\mathcal N\left( 0,\left(\frac{\sigma_{\mathrm p}}{k^r}\right)^2I \right),  \qquad r>0,
\end{equation*}
where $\sigma_p$ is a parameter we choose.

We minimize the total loss for all $B$ loops by changing their Fourier
coefficients. We apply gradient descent technique some number of steps (which we choose) and decrease the step size $\eta_s$ with each step. After
the optimization we check every loop separately. We keep only loops with positive actions and with boundary and characteristic errors  below chosen tolerance. 
\\
\\
\textbf{Stage 2: local Gaussian sampling}
\\

After stage 1, we get some loops with small boundary and characteristic errors. Now, the goal of stage 2 is to explore neighborhoods of the loops we have after stage 1.

Let $\theta_1,\ldots,\theta_R$ be the loops accepted after Stage~1. From each
loop we produce $K_{\mathrm{loc}}$ (number which we choose) new loops. One copy is left unchanged and Gaussian noise is added to the other copies.

The center coefficients receive noise with standard deviation $\tau_0$. All
already existing cosine and sine coefficients receive noise with standard
deviation $\tau$. It can be written  as
\begin{equation*}
\widetilde\theta_{i,j}=\theta_i+\xi_{i,j}, \qquad j=1,\ldots,K_{\mathrm{loc}},
\end{equation*}
where $\xi_{i,1}=0$ (we keep the original loop) and the other $\xi_{i,j}$ are Gaussian.

After making these local copies we increase the number of Fourier terms. The
old coefficients are copied and only the new terms are randomly initialized.
If $N_1$ and $N_2$ are the old and new numbers of modes, then for
$N_1<k\leq N_2$ their standard deviation is $\frac{\sigma_{\mathrm{new}}}{k^r}$,
where $\sigma_{new}$ is the number we choose. We minimize the loss for all new loops and again keep only loops with positive action and with boundary and characteristic error below some chosen number.

As in the stage one, we decrease the step size $\eta_s$ with each step. The number of steps is another parameter we need to choose.
\\
\\
\textbf{Stage 3: increasing accuracy}
\\

The purpose of this stage is to improve several stable candidates before the final optimization.

In Stage~3 we select some of the loops with the smallest action and also some
random valid loops from Stage~2. Keeping random loops prevents the method from
following only one candidate too early.

We increase the number of Fourier terms again. The old coefficients are copied
without adding new noise. Only the newly added coefficients are initialized by
Gaussian variables whose standard deviation is
$\sigma_{\mathrm{new}}/k^r$. 

The bigger number of Fourier terms lets the loops follow the boundary and characteristic direction more accurately. 

As before, we do the gradient optimization with smaller steps. 
\\
\\
\textbf{Stage 4: final optimization}
\\

In Stage 4 we keep some loops with the smallest action together with a random
selection of the remaining loops. We increase the number of Fourier
terms one more time. As before, the old coefficients are copied and only the new terms receive decaying Gaussian noise.

Note that loops in this stage already have close to zero boundary and characteristic losses. In this stage, we do not use gradient descent method. Instead, we use LBFGS. Gradient descent uses
only the current gradient. LBFGS method uses information from several previous
gradients to approximate the local curvature of the loss. This usually gives a
more accurate result when a good starting loop is already known.

After this optimization we take the one with the smallest action (bigger than a number close to zero). This gives our approximation of EHZ capacity $c_{EHZ}(K)$.
\\
\\
\textbf{Trivial loops}
\\

During this process some loops definitely shrinks to a point or are close to it. As a result, the action converges to zero. However, some loops converge to a local minimum (in the space of coefficients of the finite Fourier series). Also, after each stage, we remove loops with very small actions from our considerations 

As the table below shows, there are loops at the local minimum with actual minimum value. To avoid getting only trivial loops  we sample many initial loops and add noise to their initialization.

\subsection{Explicit form of the losses}

Let us give explicit form of the losses mentioned in the previous sections.

Recall, that we generate $B$ loops. Let $M$ be the number of points on a loop, where we evaluate values of the loop. Note that we need to increase $M$ when consider more and more terms in the Fourier series, i.e. we need more points to catch high frequencies. We assume that these $M$ points are distributed equally and have the following form:
\begin{equation*}
t_m=\frac{m}{M},\qquad m=0,\ldots,M-1.
\end{equation*}
The losses (which we minimize) are averages over both indices: the loop
index $i$ and the time index $m$.
\\
\\
\textbf{Boundary loss}
\\
\\
The loop lies on $\partial K$ when $F(\gamma_i(t))=1$. We measure the failure
of this condition by
\begin{equation*}
L_{\mathrm{bound}}(\theta) =\frac{1}{BM} \sum_{i=1}^{B}\sum_{m=0}^{M-1} \left(F(\gamma_i(t_m))-1 \right)^2.
\end{equation*}
\\
\\
\textbf{Characteristic loss}
\\
\\
Using the sign convention of the previous section, define
\begin{equation*}
X_{i,m}=-J\nabla F(\gamma_i(t_m)).
\end{equation*}
The code projects the tangent vector onto this characteristic direction. The
nonnegative projection coefficient is
\begin{equation*}
\lambda_{i,m} = \max\left\{ 0, \frac{ \left\langle\dot\gamma_i(t_m),X_{i,m} \right \rangle} {\|X_{i,m}\|^2+\varepsilon} \right\},
\end{equation*}
and the remaining vector is
\begin{equation*}
R_{i,m} = \dot\gamma_i(t_m)-\lambda_{i,m}X_{i,m}.
\end{equation*}
We define
\begin{equation*}
L_{char}(\theta) = \frac{1}{BM} \sum_{i=1}^{B}\sum_{m=0}^{M-1} \|R_{i,m}\|^2.
\end{equation*}
\\
\\
\textbf{Action loss}
\\
\\
The action loss is the mean action of all loops. We write it as
\begin{equation*}
L_{act}(\theta) = \frac{1}{2BM} \sum_{i=1}^{B}\sum_{m=0}^{M-1} \left\langle -J\gamma_i(t_m),\dot\gamma_i(t_m) \right\rangle.
\end{equation*}
Since the grid is uniform, the average over $m$ approximates the integral over $[0,1]$.
\\
\\
\textbf{Total loss}
\\
\\
The total loss is 
\begin{equation*}
L(\theta) = \alpha_{\mathrm{bound}}L_{\mathrm{bound}}(\theta) + \alpha_{\mathrm{char}}L_{\mathrm{char}}(\theta) + \alpha_{\mathrm{act}}L_{\mathrm{act}}(\theta).
\end{equation*}
Note that the total loss may be much higher than the resulting capacity. The reason is that we take average of actions over all loops and we may get different loops as local minima. Not all local minima are supposed to have the minimal action.

At the end, we take the loop with smallest action (bigger than some number close to zero).

\subsection{Test results}

Below we present numerical results of our experiments. As an additional check, we apply a random unitary transformation $U \in U(n)$ to each test domain $K$. Since the EHZ capacity is invariant under unitary transformations, the capacity should remain unchanged
\begin{equation}
c_{EHZ}(UK)=c_{EHZ}(K).
\end{equation}
We then recompute the capacity numerically for the transformed domain and compare it with the same analytical value.

The exact values used for comparison in the benchmark tests are given by standard formulas for the Ekeland--Hofer--Zehnder capacity of ellipsoids and related convex and toric domains. These formulas, together with the relevant properties of symplectic capacities, can be found in \cite{EkelandHofer1989,CieliebakHoferLatschevSchlenk,GuttHutchings,GuttHutchingsRamos}; see also \cite{HaimKislevOstrover} for closely related product constructions. We use these analytical values as reference values for evaluating the accuracy of the numerical method.
\\

Note that the exact formulas proved in this paper are first observed using numerical estimates. This also shows that the code estimates capacity numerically in an appropriate way.

\begin{figure}[ht]
    \centering
    \makebox[\textwidth][c]{%
        \includegraphics[width=1.15\textwidth]{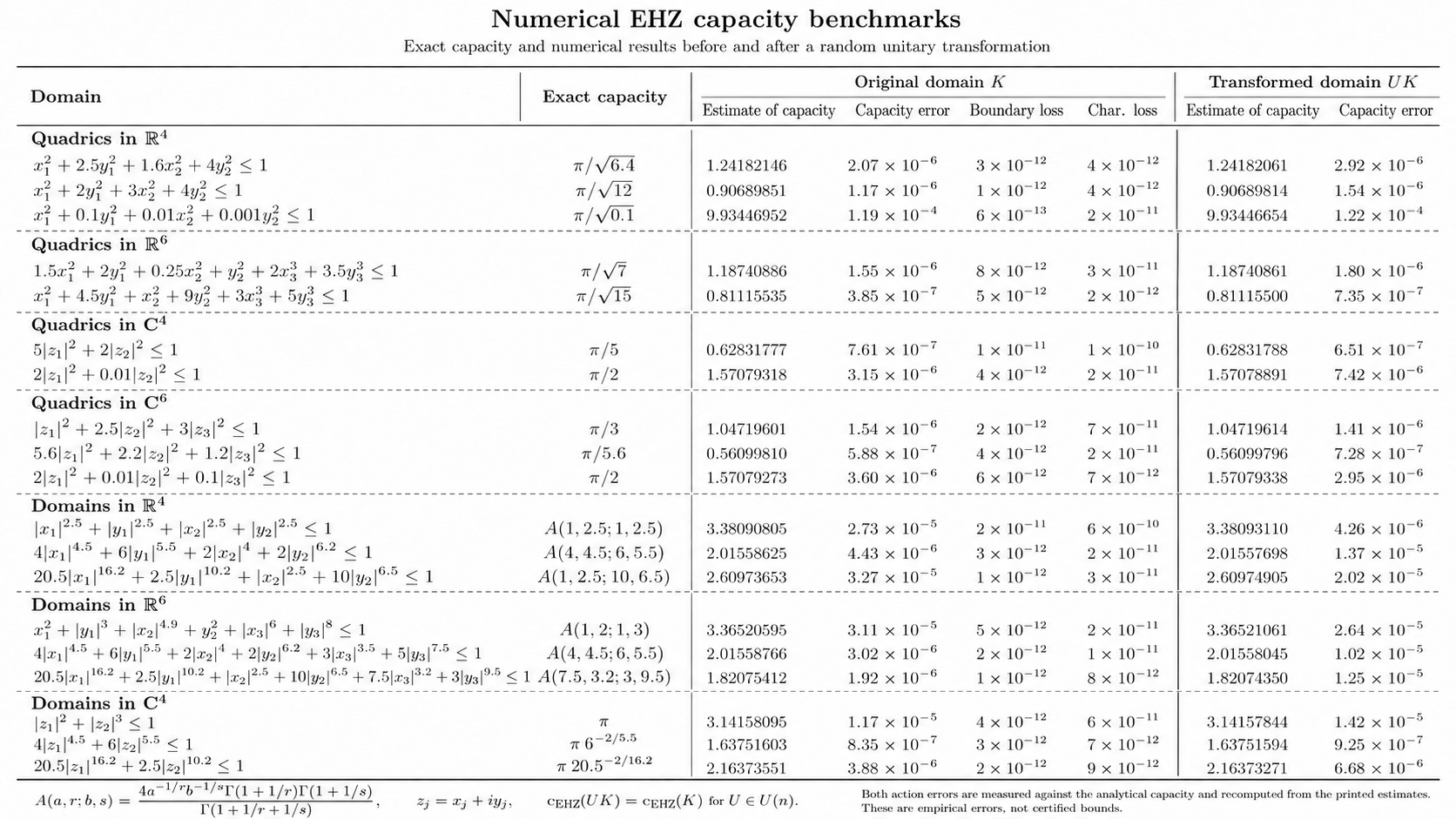}
    }
    \caption{Numerical EHZ capacity benchmarks.}
    \label{fig:ehz_numerical_tests}
\end{figure}

\end{document}